\documentclass[11pt]{article}
\usepackage[a4paper,centering,hscale=0.7,vscale=0.75]{geometry}
\usepackage{mathtools}
\usepackage{amsfonts,amsthm,amssymb}

\newtheorem{prop}{Proposition}[section]
\newtheorem{thm}[prop]{Theorem}
\newtheorem{lemma}[prop]{Lemma}

\theoremstyle{remark}
\newtheorem{rmk}[prop]{Remark}
\theoremstyle{definition}

\numberwithin{equation}{section}

\renewcommand{\P}{\mathbb{P}}
\newcommand{\E}{\mathbb{E}}

\renewcommand{\L}{\mathbb{L}}
\newcommand{\erre}{\mathbb{R}}

\renewcommand{\epsilon}{\varepsilon}
\newcommand{\m}{\bar{\mu}}

\DeclarePairedDelimiter\norm{\lVert}{\rVert}
\DeclarePairedDelimiterX\ip[2]{\langle}{\rangle}{#1,#2}

\title{On maximal inequalities for purely discontinuous $L_q$-valued
  martingales}

\author{Carlo Marinelli\thanks{Department of Mathematics, University
    College London, Gower Street, London WC1E 6BT, United Kingdom.}}

\date{\normalsize 27 November 2013}

\begin{document}
\maketitle

\begin{abstract}
  We prove maximal inequalities for $L_q$-valued martingales obtained
  by stochastic integration with respect to compensated random
  measures. A version of these estimates for integrals with respect to
  compensated Poisson random measures were first obtained by Dirksen
  \cite{Dirksen} using arguments based on inequalities for sum of
  independent Banach-space-valued random variables, geometric
  properties of Banach spaces, and decoupling inequalities. Our proofs
  are completely different and rely almost exclusively on classical
  stochastic analysis for real semimartingales.
\end{abstract}

\section{Introduction}
Our purpose is to prove maximal inequalities for purely discontinuous
martingales taking values in $L_q$ spaces, with bounds expressed
in terms of predictable elements only. In particular, if $\mu$ is a
random measure with compensator $\nu$ and $\m:=\mu-\nu$, the main
result takes the form
\[
\Bigl( \E\sup_{t \geq 0} \norm[\big]{(g \star \m)_t}^p_{L_q} \Bigr)^{1/p}
\eqsim_{p,q} \norm[\big]{g}_{\mathcal{I}_{p,q}}
\qquad \forall p,\,q \in \mathopen]1,\infty\mathclose[,
\]
where the integrand $g$ takes values in an $L_q$ space, and the norm
on the right-hand side depends, roughly speaking, on integrals of
functionals of $g$ with respect to $\nu$ only.

Estimates of this type were obtained for the first time, assuming that
$\mu$ is a Poisson random measure, by Dirksen \cite{Dirksen}, using an
abstract (and very elegant) approach. Namely, he first obtains
suitable vector-valued generalizations of Rosenthal's inequality for
sums of independent random variables, and then deduces from them
inequalities for stochastic integrals of step processes with respect
to compensated Poissonian measures by means of decoupling
techniques. His proofs rely on several sophisticated arguments,
pertaining to the interplay between the geometry of Banach spaces and
estimates for series of random variables on them, as well as, as
already mentioned, to (vector-valued extensions of) decoupling
inequalities.

Our approach is completely different and uses essentially only
stochastic calculus for \emph{real} semimartingales. Our interest in
this problem arose while working on \cite{cm:EJP10}, where we derived
a very special case of the above maximal inequality (namely the case
$p=q\geq 2)$, by a rather elementary integration in space of a
corresponding inequality for real-valued processes (apparently) due to
Novikov \cite{Nov:75}. The latter inequality has been known for almost
40 years (cf.~\cite{cm:BJ-surv} for an extensive review and a brief
historical account). However, it seems to be generally believed that,
as common wisdom suggests, this naive approach is doomed to fail if
one wants to estimate the $p$-th moment of the $L_q$-norm, with $p
\neq q$. This is actually one of the main reasons, at least from the
point of view of someone interested in stochastic PDEs, why developing
stochastic integration and, more generally, stochastic calculus on
Banach spaces is a worthy endeavor (see e.g. \cite{vNVW:max} and
references therein for recent developments in this direction, when the
integrator is a Wiener process). One of the main messages of the
present work is that the common belief just described is exaggerated,
in the sense that one can indeed get very far using only pointwise
estimates, integration, and classical results of stochastic
calculus. In a figurative way, thinking to a scale of abstraction's
level, one could say that Dirksen's results are obtained ``from
above'', and ours are obtained ``from below''.

\section{Main result}
Let $(X,\mathcal{A},n)$ be a measure space, and denote $L_q$ spaces on
$X$ simply by $L_q$, for any $q \geq 1$.  
All random elements will be defined on a fixed filtered probability
space $(\Omega,\mathcal{F},\mathbb{F},\P)$,
$\mathbb{F}:=(\mathcal{F}_t)_{t \geq 0}$. We shall use the notation
$\norm{\cdot}_{\L_p}$ to denote
$\bigl(\E|\cdot|^p\bigr)^{1/p}$. Similarly, given a normed space $E$
and an $E$-valued random variable $\xi$, we shall use the notation
$\norm{\xi}_{\L_p(E)}:=\bigl(\E\norm{\xi}_E^p\bigr)^{1/p}$, sometimes
omitting the parentheses around $E$ if the meaning is clear. A
completely similar convention will be in place also for other
integrals on other spaces. Let $\mu$ be a random measure on $\erre_+
\times Z$, $\nu$ be its dual predictable projection (compensator), and
$\m:=\mu-\nu$ be the corresponding compensated random
measure. Integrals over $\erre_+ \times Z$ will be denoted simply by
an integration sign. Let $g: \Omega \times \erre_+ \times Z \times X
\to \erre$ be such that $(\omega,t,z) \mapsto g(\omega,t,z,x)$ is
predictable for each $x$ and $(\omega,t,z) \mapsto
\norm[\big]{g(\omega,t,z,\cdot)} \in L_q$ for all $(\omega,t,z)$.

For any $p_1,p_2,p_3 \in [1,\infty]$, introduce the spaces
\[
L_{p_1,p_2,p_3} := \L_{p_1}L_{p_2}(\nu)L_{p_3},
\qquad
\tilde{L}_{p_1,p_2} := \L_{p_1}L_{p_2}L_2(\nu),
\]
where $L_p(\nu):=L_p(\erre_+ \times Z,\nu)$ and the above convention
about $L_p$ spaces with mixed norms is in place. In particular, this
means that, for instance,
\[
\xi \in L_{p_2}(\nu)L_{p_3} \quad \Leftrightarrow \quad
\norm[\big]{\xi}_{L_{p_2}(\nu)L_{p_3}} := \biggl(
\int_{\erre_+ \times Z} \norm{\xi}^{p/2}_{L_{p_3}(X)} \,d\nu
\biggr)^{1/p_2} < \infty.
\]

The (proof of the) following theorem, whose original formulation is due
to Dirksen \cite{Dirksen}, is our main result.
\begin{thm}     \label{thm:m}
  Let $p$, $q \in \mathopen]1,\infty\mathclose[$. One has
  \[
  \norm[\Big]{\sup_{t \geq 0} \norm{(g \star \m)_t}_{L_q}}_{\L_p}
  \eqsim_{p,q} \norm{g}_{\mathcal{I}_{p,q}},
  \]
  where
  \begin{equation}     \label{eq:ipq}
  \mathcal{I}_{p,q} := 
  \begin{cases}
    L_{p,p,q} + L_{p,q,q} + \tilde{L}_{p,q}, &\quad 1 < p \leq q \leq 2,\\
    (L_{p,p,q} \cap L_{p,q,q}) + \tilde{L}_{p,q}, &\quad 1 < q \leq p \leq 2,\\
    L_{p,p,q} \cap (L_{p,q,q} + \tilde{L}_{p,q}), &\quad 1 < q < 2 \leq p,\\
    L_{p,p,q} + (L_{p,q,q} \cap \tilde{L}_{p,q}), &\quad 1 < p < 2 \leq q,\\
    (L_{p,p,q} + L_{p,q,q}) \cap \tilde{L}_{p,q}, &\quad 2 \leq p \leq q,\\
    L_{p,p,q} \cap L_{p,q,q} \cap \tilde{L}_{p,q}, &\quad 2 \leq q \leq p.
  \end{cases}
\end{equation}
\end{thm}
The proof of theorem, which follows by a series of Lemmata and Propositions,
is in Section \ref{sec:p} below.

An explicit description of the norms of the spaces appearing above is
as follows:
\begin{align*}
  \norm[\big]{g}_{L_{p,p,q}} &=
    \left( \E\int \norm[\big]{g}^p_{L_q}\,d\nu \right)^{1/p},\\
  \norm[\big]{g}_{L_{p,q,q}} &=
    \left( \E\biggl(\int \norm[\big]{g}^q_{L_q}\,d\nu\biggr)^{p/q} \right)^{1/p},\\
  \norm[\big]{g}_{\tilde{L}_{p,q}} &= 
    \left( \E\norm[\bigg]{\biggl(\int |g|^2\,d\nu\biggr)^{1/2}}^p_{L_q}
    \right)^{1/p}.
\end{align*}

\section{Preliminaries and auxiliary results}
The maximal inequalities in the following theorem are known, and their
proofs can be found e.g. in \cite{cm:SEE2,cm:JFA10,cm:BJ-surv}.
\begin{thm}     \label{thm:H}
  Let $g$ take values in a Hilbert space $H$. Then one has
  \begin{align}
    \label{eq:sqh}
    \E \sup_{t \geq 0} \norm[\big]{(g \star \m)_t}_H^p &\lesssim_p
    \E \biggl( \int \norm{g}_H^2 \,d\nu \biggr)^{\frac12 p}
    &\qquad \forall p \in \left]0,2\right],\\
    \label{eq:mejeto}
    \E \sup_{t \geq 0} \norm[\big]{(g \star \m)_t}_H^p &\lesssim_p
    \E \int \norm{g}_H^p \,d\nu
    &\qquad \forall p \in \left[1,2\right],\\
    \label{eq:mejo}
    \E \sup_{t \geq 0} \norm[\big]{(g \star \m)_t}_H^p &\lesssim_p
    \E \int \norm{g}_H^p \,d\nu
    + \E \biggl(\int \norm{g}_H^2 \,d\nu\biggr)^{\frac12 p}
    &\qquad \forall p \in \left[2,\infty\right[.
  \end{align}
\end{thm}

\medskip

Throughout the rest of the paper, unless otherwise stated, we shall
use the notation $M:=g \ast \m$, as well as
\[
[M,M]_t := \int_0^t\!\!\!\int_Z |g|^2\,d\mu,
\qquad
\ip{M}{M}_t := \int_0^t\!\!\!\int_Z |g|^2\,d\nu.
\]
Moreover, for notational compactness, we shall write, for any
$L_q$-valued process $Y$,
\[
\norm[\big]{Y^*_\infty}_{\L_pL_q}
:= \Bigl( \E\sup_{t \geq 0} \norm{Y_t}_{L_q}^p \Bigr)^{1/p}.
\]

\medskip

The following estimate plays a crucial role throughout the paper.
\begin{thm}     \label{thm:iBDG}
  One has, for any $p$, $q \in \mathopen]1,\infty\mathclose[$,
  \begin{equation}
    \label{eq:iBDG}
    \norm[\big]{[M,M]^{1/2}_\infty}_{\L_pL_q} \lesssim_{p,q}
    \norm[\big]{M^*_\infty}_{\L_pL_q} 
    \lesssim_{p,q} \norm[\big]{[M,M]^{1/2}_\infty}_{\L_pL_q}.
  \end{equation}
  Moreover, the upper bound also holds for $p=q=1$.
\end{thm}
\begin{proof}
  The map $M \mapsto [M,M]_T^{1/2}$ is sublinear and bounded on
  $\L_pL_p$ and on $\L_pL_2$ for all $p \in
  \mathopen]1,\infty\mathclose[$. The inequality on the left then
  follows by the extension of Riesz-Thorin interpolation due to
  Benedek and Panzone \cite{BenPan}, coupled with the linearization
  theorem by Janson \cite{Jans:interp}.

  To prove the inequality on the right, we use a duality argument: we
  have
  \[
  \norm[\big]{M_\infty}_{\L_pL_q} = 
  \sup_{\zeta \in B_1(\L_{p'}L_{q'})} \E(M_\infty,\zeta),
  \]
  where $(\cdot,\cdot)$ stands for the duality form between $L_p$ and
  $L_{p'}$. Now take a martingale $N$ with final value
  $N_\infty=\zeta$, and note
  \begin{align*}
    \E(M_\infty,\zeta) &= \E(M_\infty,N_\infty) = \E[M_\infty,N_\infty]\\
    &\leq \norm[\big]{[M,M]^{1/2}_\infty}_{\L_pL_q} 
          \norm[\big]{[N,N]^{1/2}_\infty}_{\L_{p'}L_{q'}}\\
    &\lesssim_{p,q} \norm[\big]{[M,M]^{1/2}_\infty}_{\L_pL_q}
                   \norm[\big]{N_\infty}_{\L_{p'}L_{q'}}
    \leq \norm[\big]{[M,M]^{1/2}_\infty}_{\L_pL_q},
  \end{align*}
  from which the second inequality follows immediately.
\end{proof}

\medskip

The following simple estimate will be used repeatedly.
\begin{lemma}     \label{lm:stimette}
  Let $(X,\mathcal{A},m)$ be a measure space and $p > r \geq 1$. If $f
  \in L_r(X,m) \cap L_p(X,m)$, then $f \in L_q(X,m)$ for all $q \in
  \mathopen]r,\mathclose p[$, and
  \[
  \norm{f}^\alpha_{L_q} \leq \norm{f}^\alpha_{L_r} + \norm{f}^\alpha_{L_p}
  \qquad \forall \alpha>0.
  \]
\end{lemma}
\begin{proof}
  Let $\theta \in [{0,1}[$ be such that $q=r\theta+p(1-\theta)$. By
  Lyapunov's inequality one has, after raising to the power $\alpha$,
  \[
  \norm{f}^\alpha_{L_q} \leq \norm{f}_{L_r}^{\alpha\theta} \, 
                           \norm{f}_{L_p}^{\alpha(1-\theta)}.
  \]
  Young's inequality $ab \leq a^s/s + b^{s'}/s'$ with conjugate
  exponents $s=1/\theta$ and $s'=1/(1-\theta)$ yields
  \[
  \norm{f}^\alpha_{L_q} \leq \theta \norm{f}_{L_r}^\alpha + 
                           (1-\theta) \norm{f}_{L_p}^{\alpha}
  \leq \norm{f}_{L_r}^\alpha + \norm{f}_{L_p}^{\alpha}.
  \qedhere
  \]
\end{proof}

\medskip

We shall also use several times the following inequality between norms
of functions in $L_p$-spaces with mixed norms, which sometimes goes under
the name of H\"older-Minkowski's inequality:
\begin{equation}     \label{eq:HM}
\norm[\big]{f}_{L_p(L_q)} \leq \norm[\big]{f}_{L_q(L_p)}
\qquad \forall p \geq q.
\end{equation}


\section{Proof of the main result}     \label{sec:p}
It is enough to prove only the upper bounds
\[
\norm[\big]{(g \ast \m)_\infty^*}_{\L_pL_q} 
\lesssim_{p,q} \norm[\big]{g}_{\mathcal{I}_{p,q}},
\]
as the lower bounds will follow by duality (in fact, the dual of
$\mathcal{I}_{p,q}$ is $\mathcal{I}_{p',q'}$ -- cf.~\cite{Dirksen}).

It should be noted that in many cases we prove in fact more general
results than those needed to obtain the above upper bound.

\subsection{Case $1 < p \leq q \leq 2$}
The upper bound in Theorem \ref{thm:m} with parameters $p$ and $q$
such that $1 < p \leq q \leq 2$ is a consequence of the next three
Propositions.

\begin{prop}     \label{prop:uno}
  Let $1 \leq q \leq 2$, $0 < p \leq q$. One has
  \begin{equation*}
    \E\sup_{t\geq 0} \norm[\big]{(g \star \m)_t}_{L_q}^p \lesssim_{p,q}
    \E\biggl( \int
    \norm[\big]{g}_{L_q}^q\,d\nu \biggr)^{p/q}.
  \end{equation*}
\end{prop}
\begin{proof}
  Inequality \eqref{eq:mejeto} with exponent $1 \leq q \leq
  2$ and $H=\erre$ yields
  \[
  \E\sup_{t \geq 0} \big\lvert (g \star \m)_t \big\rvert^q \lesssim_q
  \E\int |g|^q\,d\nu,
  \]
  hence also, by Fatou's lemma and Tonelli's theorem,
  \[
  \E\sup_{t\geq 0} \norm[\big]{(g \star \m)_t}_{L_q}^q \lesssim_q
  \E \int \norm[\big]{g}_{L_q}^q\,d\nu.
  \]
  Let $T$ be any stopping time. Replacing $g$ with
  $g\mathbf{1}_{[0,T]}(t)$, the previous inequality implies
  \[
  \E\sup_{t\leq T} \norm[\big]{(g \star \m)_t}_{L_q}^q \lesssim_q
  \E \int_0^T\!\!\!\int_Z \norm[\big]{g}_{L_q}^q\,d\nu.
  \]
  Lenglart's domination inequality finally gives
  \[
  \E\sup_{t\geq 0} \norm[\big]{(g \star \m)_t}_{L_q}^p \lesssim_{p,q}
  \E\biggl( \int \norm[\big]{g}_{L_q}^q\,d\nu \biggr)^{p/q}
  \]
  for any $0 < p < q$.
\end{proof}
\begin{rmk}
  The localization step spelled out in the previous proof, which is
  needed to apply Lenglart's domination inequality, will be implicitly
  assumed in the proofs to come.
\end{rmk}

\begin{prop}
  Let $0 < p \leq q \leq 2$. One has
  \begin{equation*}
  \E\sup_{t \geq 0} \norm[\big]{(g \star \m)_t}_{L_q}^p \lesssim_{p,q}
  \E\norm[\bigg]{\biggl(\int |g|^2\,d\nu\biggr)^{1/2}}_{L_q}^p.
  \end{equation*}
\end{prop}
\begin{proof}
  Inequality \eqref{eq:sqh} with exponent $q \in
  \mathopen]0,2\mathclose]$ and $H=\erre$ yields
  \[
  \E \sup_{t \geq 0} \big\lvert (g \star \m)_t \big\rvert^q \lesssim_q
  \E\Bigl( \int |g|^2\,d\nu \Bigr)^{q/2}.
  \]
  Integrating over $X$, taking into account Fatou's lemma and
  Tonelli's theorem, one obtains
  \[
  \E\sup_{t \geq 0} \norm[\big]{(g \star \m)_t}_{L_q}^q
  \lesssim_q
  \E\norm[\bigg]{\biggl(\int_0^T |g|^2\,d\nu\biggr)^{1/2}}_{L_q}^q,
  \]
  which in turn yields, appealing to Lenglart's domination
  inequality,
  \[
  \E\sup_{t \geq 0} \norm[\big]{(g \star \m)_t}_{L_q}^p \lesssim_{p,q}
  \E\norm[\bigg]{\biggl(\int |g|^2\,d\nu\biggr)^{1/2}}_{L_q}^p.
  \qedhere
  \]
\end{proof}

\begin{prop}     \label{prop:13}
  Let $1 < p \leq q \leq 2$. One has
  \begin{equation*}
    \E\sup_{t\geq 0} \norm[\big]{(g \star \m)_t}_{L_q}^p \lesssim_{p,q}
    \E\int \norm[\big]{g}_{L_q}^p\,d\nu.
  \end{equation*}
\end{prop}
\begin{proof}
  By Theorem \ref{thm:iBDG} one has
  \[
  \norm[\big]{(g \star \m)^*_\infty}_{\L_pL_q} \lesssim_{p,q}
  \norm[\big]{[M,M]_\infty^{1/2}}_{\L_pL_q} =
  \norm[\big]{\norm{\Delta M}_{\ell_2}}_{\L_pL_q} =
  \norm[\big]{\Delta M}_{\L_pL_q\ell_2}.
  \]
  Since $p \leq 2$, one has $\norm{\Delta M}_{\ell_2} \leq
  \norm{\Delta M}_{\ell_p}$, hence, by inequality \eqref{eq:HM},
  \[
  \norm[\big]{\Delta M}^p_{\L_pL_q\ell_2} \leq
  \norm[\big]{\Delta M}^p_{\L_pL_q\ell_p} \leq 
  \norm[\big]{\Delta M}^p_{\L_p\ell_pL_q} =
  \E \int \norm{g}_{L_q}^p\,d\mu = \E \int \norm{g}_{L_q}^p\,d\nu.
  \qedhere
  \]
\end{proof}


\subsection{Case $1 < q \leq p \leq 2$}
The upper bound in Theorem \ref{thm:m} with parameters $p$ and $q$
such that $1 < q \leq p \leq 2$ is a consequence of the next two
Propositions.

\begin{prop}     \label{lm:34}
  Let $1 < q \leq p \leq 2$. Then one has
  \[
  \E\sup_{t \geq 0} \norm[\big]{(g \star \m)_t}_{L_q}^p \lesssim_{p,q}
  \E\int \norm[\big]{g}_{L_q}^p\,d\nu
  + \E\biggl(\int \norm[\big]{g}_{L_q}^q\,d\nu \biggr)^{p/q}.  
  \]
\end{prop}
\begin{proof}
  Appealing to Theorem \ref{thm:iBDG}, one has 
  \[
  \E\sup_{t \geq 0} \norm[\big]{(g \star \m)}_{L_q}^p \lesssim_{p,q}
  \E\norm[\big]{[M,M]_\infty^{1/2}}_{L_q}^p =
  \norm[\big]{\Delta M}_{\L_pL_q\ell_2}^p,
  \]
  where, since $q \leq 2$,
  \[
  \norm[\big]{\Delta M}_{\L_pL_q\ell_2}^p \leq 
  \norm[\big]{\Delta M}_{\L_pL_q\ell_q}^p =
  \norm[\big]{\Delta M}_{\L_p\ell_qL_q}^p = 
  \E\biggl(\int \norm[\big]{g}_{L_q}^q\,d\mu \biggr)^{p/q}.
  \]
  Writing
  \[
  \int \norm[\big]{g}_{L_q}^q\,d\mu = \int \norm[\big]{g}_{L_q}^q\,d\m
  + \int \norm[\big]{g}_{L_q}^q\,d\nu,
  \]
  the previous two inequalities yield
  \begin{align*}
  \E\sup_{t \geq 0} \norm[\big]{(g \star \m)}_{L_q}^p
  &\lesssim_{p,q} \E\biggl|\int \norm[\big]{g}_{L_q}^q\,d\m \biggr|^{p/q}
  + \E\biggl(\int \norm[\big]{g}_{L_q}^q\,d\nu \biggr)^{p/q}\\
  &\lesssim \E\int \norm[\big]{g}_{L_q}^p\,d\nu
  + \E\biggl(\int \norm[\big]{g}_{L_q}^q\,d\nu \biggr)^{p/q},
  \end{align*}
  where we have used inequality \eqref{eq:mejo} with exponent $p/q
  \leq 2$.
\end{proof}

\begin{prop}     \label{prop:rifatta}
  Let $1 < q \leq p \leq 2$. Then one has
  \[
  \E\sup_{t \geq 0} \norm[\big]{(g \star \m)_t}_{L_q}^p \lesssim_{p,q}
  \E\norm[\bigg]{\biggl( \int |g|^2\,d\nu \biggr)^{1/2}}_{L_q}^p
  \]
\end{prop}
For the proof we need the following estimate.
\begin{lemma}
  Let $2 \leq p \leq q$. Then
  \[
  \norm[\big]{\ip{M}{M}_\infty^{1/2}}_{\L_pL_q} \lesssim_{p,q} 
  \norm[\big]{M_\infty}_{\L_pL_q}.
  \]
\end{lemma}
\begin{proof}
  One has
  \begin{align*}
    \E\norm[\big]{\ip{M}{M}_\infty^{1/2}}^p_{L_q} &=
    \E\norm[\big]{\ip{M}{M}_\infty}^{p/2}_{L_{q/2}}\\
    &= \E \norm[\bigg]{\int |g|^2\,d\nu}^{p/2}_{L_{q/2}}
    \lesssim \E \norm[\bigg]{\int |g|^2\,d\m}^{p/2}_{L_{q/2}}
    + \E \norm[\bigg]{\int |g|^2\,d\mu}^{p/2}_{L_{q/2}},
  \end{align*}
  where, thanks to Theorem \ref{thm:iBDG} and Doob's inequality,
  \[
  \E \norm[\bigg]{\int |g|^2\,d\mu}^{p/2}_{L_{q/2}} =
  \E \norm[\big]{[M,M]_\infty}^{p/2}_{L_{q/2}} =
  \E \norm[\big]{[M,M]^{1/2}_\infty}^p_{L_q} \lesssim_{p,q}
  \E \norm[\big]{M_\infty}^p_{L_q}.
  \]
  Moreover, setting $N:=|g|^2 \ast \m$, one has, again by Theorem
  \ref{thm:iBDG},
  \begin{align*}
  \E \norm[\bigg]{\int |g|^2\,d\m}^{p/2}_{L_{q/2}} &\lesssim_{p,q}
  \E \norm[\big]{[N,N]_\infty^{1/2}}^{p/2}_{L_{q/2}}
  = \E \norm[\bigg]{\biggl( \int |g|^4\,d\mu \biggr)^{1/2}}^{p/2}_{L_{q/2}}\\
  &= \E \norm*{\Bigl(\sum \norm{\Delta M}^4\Bigr)^{1/2}}^{p/2}_{L_{q/2}}
  = \E \norm*{\norm[\big]{\Delta M}_{\ell_4}^2}^{p/2}_{L_{q/2}}\\
  &= \E \norm[\Big]{\norm[\big]{\Delta M}_{\ell_4}}^p_{L_q}
  \leq \E \norm[\Big]{\norm[\big]{\Delta M}_{\ell_2}}^p_{L_q}\\
  &= \E \norm[\big]{[M,M]_\infty^{1/2}}^p_{L_q}
  \lesssim_{p,q} \E \norm[\big]{M_\infty}^p_{L_q}.
  \qedhere
\end{align*}
\end{proof}

\begin{proof}[Proof of Proposition \ref{prop:rifatta}]
  We use a duality argument: let us write
  \[
  \norm[\big]{M_\infty}_{\L_pL_q} = 
  \sup_{N_\infty\in B_1} \E \int_X M_\infty N_\infty,
  \]
  where $B_1$ stands for the unit ball of $\L_{p'}L_{q'}$. Introduce
  the martingale $N$ defined by $N_t=\E[N_\infty|\mathcal{F}_t]$ for
  all $t \geq 0$. The identity $\E M_\infty N_\infty = \E
  \ip{M}{N}_\infty$, Kunita-Watanabe's inequality, H\"older's
  inequality, and the previous Lemma imply
  \begin{align*}
    \E\int_X M_\infty N_\infty &\leq \norm[\big]{\ip{M}{M}_\infty^{1/2}}_{\L_pL_q} 
    \, \norm[\big]{\ip{N}{N}_\infty^{1/2}}_{\L_{p'}L_{q'}}\\
    &\lesssim_{p,q} \norm[\big]{\ip{M}{M}_\infty^{1/2}}_{\L_pL_q} \,
          \norm[\big]{N_\infty}_{\L_{p'}L_{q'}}
    \leq \norm[\big]{\ip{M}{M}_\infty^{1/2}}_{\L_pL_q},
  \end{align*}
  whence the conclusion, because
  \[
  \ip{M}{M}_\infty = \int |g|^2\,d\nu.
  \qedhere
  \]
\end{proof}


\subsection{Case $1 < p \leq 2 \leq q$}
The upper bound in Theorem \ref{thm:m} with parameters $p$ and $q$
such that $1 < p \leq 2 \leq q$ follows by the next two Propositions.

\begin{prop}     \label{prop:bella}
  Let $q \geq 2$, $0 < p \leq q$. Then one has
  \[
  \E \sup_{t \leq T} \norm[\big]{(g \star \m)_t}_{L_q}^p
  \lesssim_{p,q} \E \biggl( \int \norm{g}_{L_q}^q \,d\nu \biggr)^{p/q}
  + \E \norm[\bigg]{\biggl( \int |g|^2 \,d\nu \biggr)^{1/2}}^p_{L_q}.
  \]
\end{prop}
\begin{proof}
  Inequality \eqref{eq:mejo}, with exponent $q \geq 2$ and $H=\erre$, and
  integration over $X$ yield
  \begin{equation}     \label{eq:mejo_Lp}
  \E \sup_{t \leq T} \norm[\big]{(g \star \m)_t}_{L_q}^q
  \lesssim_q \E \int \norm{g}_{L_q}^q \,d\nu
  + \E \norm[\bigg]{\biggl( \int |g|^2 \,d\nu \biggr)^{1/2}}^q_{L_q},
  \end{equation}
  therefore, by Lenglart's domination inequality,
  \begin{equation*}
  \E \sup_{t \leq T} \norm[\big]{(g \star \m)_t}_{L_q}^p
  \lesssim_{p,q} \E \biggl( \int \norm{g}_{L_q}^q \,d\nu \biggr)^{p/q}
  + \E \norm[\bigg]{\biggl( \int |g|^2 \,d\nu \biggr)^{1/2}}^p_{L_q}.
  \qedhere
  \end{equation*}
\end{proof}

\begin{prop}     \label{prop:due}
  Let $1 < p \leq 2 \leq q$. Then one has
  \[
  \E\sup_{t \geq 0} \norm[\big]{(g \ast \m)_t}_{L_q}^p \lesssim_{p,q}
  \E\int \norm{g}^p_{L_q}\,d\nu.
  \]
\end{prop}
\begin{proof}
  By Theorem \ref{thm:iBDG} and H\"older-Minkowski's inequality
  \eqref{eq:HM}, one has
  \[
  \norm[\big]{(g \ast \m)^*_\infty}_{\L_pL_q} \lesssim_{p,q}
  \norm[\big]{[M,M]_\infty^{1/2}}_{\L_pL_q} =
  \norm[\big]{\Delta M}_{\L_pL_q\ell_2}
  \leq \norm[\big]{\Delta M}_{\L_p\ell_2L_q} 
  \leq \norm[\big]{\Delta M}_{\L_p\ell_pL_q},
  \]
  where
  \[
  \norm[\big]{\Delta M}_{\L_p\ell_pL_q}^p =
  \E\sum \norm[\big]{\Delta M}_{L_q}^p =
  \E\int \norm[\big]{g}_{L_q}^p\,d\mu =
  \E\int \norm[\big]{g}_{L_q}^p\,d\nu.
  \qedhere
  \]
\end{proof}


\subsection{Case $1 < q \leq 2 \leq p$}
The upper bound in Theorem \ref{thm:m} with parameters $p$ and $q$
such that $1 < q \leq 2 \leq p$ is a consequence of the next two
Propositions.

\begin{prop}     \label{lm:38}
  Let $1 < q \leq 2 \leq p$. Then one has
  \[
  \E\sup_{t \geq 0} \norm[\big]{(g \star \m)_t}_{L_q}^p \lesssim_{p,q}
  \E\biggl( \int \norm{g}_{L_q}^q\,d\nu \biggr)^{p/q}
  + \E \int \norm{g}_{L_q}^p\,d\nu.
  \]
\end{prop}
\begin{proof}
  Proceeding as in the proof of Proposition \ref{lm:34}, one obtains
  \[
  \E\norm[\big]{(g \star \m)_\infty^*}_{L_q}^p \lesssim_{p,q}
  \E\biggl( \int \norm[\big]{g}_{L_q}^q \,d\mu \biggr)^{p/q}
  \lesssim \E\biggl| \int \norm[\big]{g}_{L_q}^q \,d\m \biggr|^{p/q}
  + \E\biggl( \int \norm[\big]{g}_{L_q}^q \,d\nu \biggr)^{p/q}.
  \]
  If $p \leq 2q$, i.e. if $p/q \leq 2$, the second inequality in
  Theorem \ref{thm:H} (with $H=\erre$) yields
  \[
  \E\biggl| \int \norm[\big]{g}_{L_q}^q \,d\m \biggr|^{p/q}
  \lesssim_{p,q} \E\int \norm[\big]{g}_{L_q}^p \,d\nu,
  \]
  as desired. Otherwise, if $p > 2q$, i.e. if $p/q>2$, the third
  inequality in Theorem \ref{thm:H} (with $H=\erre$) yields
  \begin{align*}
    \E\biggl| \int \norm[\big]{g}_{L_q}^q \,d\m \biggr|^{p/q}
    &\lesssim \E\int \norm[\big]{g}_{L_q}^p \,d\nu
    + \E\biggl( \int \norm[\big]{g}_{L_q}^{2q}\,d\nu \biggr)^{p/(2q)}\\
    &= \E\int \norm[\big]{g}_{L_q}^p \,d\nu
    + \E\norm*{\norm[\big]{g}_{L_q}}^p_{L_{2q}(\nu)}.
  \end{align*}
  Since we have $q < 2q < p$, by Lemma \ref{lm:stimette} one
  immediately gets
  \begin{align*}
  \E\norm*{\norm[\big]{g}_{L_q}}^p_{L_{2q}(\nu)} &\leq
  \E\norm*{\norm[\big]{g}_{L_q}}^p_{L_{q}(\nu)}
  + \E\norm*{\norm[\big]{g}_{L_q}}^p_{L_{p}(\nu)}\\
  &= \E\biggl( \int \norm[\big]{g}_{L_q}^q \,d\nu \biggr)^{p/q}
  + \E\int \norm[\big]{g}_{L_q}^p \,d\nu,
  \end{align*}
  and the proof is concluded.
\end{proof}

\begin{prop}     \label{prop:q2p}
  Let $1 \leq q \leq 2 \leq p$. Then one has
  \[
  \E\sup_{t \geq 0} \norm[\big]{(g \star \m)_t}_{L_q}^p \lesssim_{p,q}
  \E\norm[\bigg]{\biggl( \int |g|^2\,d\nu \biggr)^{1/2}}_{L_q}^p
  + \E \int \norm{g}_{L_q}^p\,d\nu.
  \]
\end{prop}

For the proof we need the following result by Lenglart, Lepingle and
Pratelli (see \cite[Lemma 1.1]{LeLePr}).
\begin{lemma}
  Let $A$ and $B$ be increasing adapted processes. If there exist
  $r>0$ and $\alpha>0$ such that
  \[
  \E(A_{T-}-A_{S-})^r \leq \alpha \E B_{T-}^r \mathbf{1}_{\{S<T\}}
  \]
  for all stopping times $S$ and $T$ such that $S<T$, then one has,
  for any moderate function $F$,
  \[
  \E F(A_\infty) \lesssim_{r,\alpha,F} \E F(B_\infty).
  \]
\end{lemma}

\begin{proof}[Proof of Proposition \ref{prop:q2p}]
  We shall proceed in several steps.

  \textsc{Step 1.} We introduce the Davis decomposition of $M:=g \star
  \m$: define the real-valued process $S_t:=\sup_{s \leq t}
  \norm{\Delta M_s}_{L_q}$, and set
  \[
  K^1_t := \sum_{\norm{\Delta M_s}>2S_{t-}} \Delta M_s,
  \qquad K^2:= \widetilde{K^1}, \qquad K:=K^1-K^2.
  \]
  Then there exists an $L_q$-valued predictable process $g'$ such that
  $K^1 = g' \star \mu$, hence $K^2 = g' \star \nu$ and $K= g' \star
  \m$. Now set $L=M-K = (g-g') \star \m$.

  \smallskip

  \textsc{Step 2.} Denoting the total variation by
  $\norm{\cdot}_{TV}$, one has
  \[
  M^*_\infty \leq K^*_\infty + L^*_\infty \leq \norm{K}_{TV} + L^*_\infty,
  \]
  hence, for any $p \geq 1$,
  \[
  \E\sup_{t \geq 0} \norm{M_t}_{L_q}^p \lesssim_p \E\norm{K}^p_{TV(L_q)}
  + \E\sup_{t \geq 0} \norm{L_t}_{L_q}^p,
  \]
  where
  \[
  \E\norm{K}^p_{TV} \lesssim_p \E\norm{K^1}^p_{TV} 
  + \E\norm{\widetilde{K^1}}^p_{TV}
  \lesssim \E\norm{K^1}^p_{TV},
  \]
  and $\norm{K^1}_{TV(L_q)} \lesssim S_\infty$ implies
  \[
  \E\norm{K}^p_{TV} \lesssim_p \E S_\infty^p.
  \]

  \smallskip

  \textsc{Step 3.} Let $S$, $T$ be any stopping times. Setting
  \[
  (L^{S,T})_t := \bigl(L_{(S+t) \wedge T} - L_{S-}\bigr) \mathbf{1}_{\{S<T\}},
  \]
  we are going to show that
  \[
  L_{T-}^* - L_{S-}^* \leq (L^{S,T})_\infty^* \mathbf{1}_{\{S<T\}}.
  \]
  Since $t \mapsto L^*_t$ is increasing and the right-hand side is
  positive, we can assume $S<T$ without loss of generality. If
  $L^*_{T-}=L^*_{S-}$, the inequality is obviously true, hence we can
  assume $L^*_{T-} > L^*_{S-}$, thus also
  \[
  L^*_{T-} \leq \sup_{S \leq s \leq T} \norm{L_s}_{L_q}.
  \]
  We therefore have, writing $\norm{\cdot}$ instead of
  $\norm{\cdot}_{L_q}$ for compactness of notation,
  \begin{align*}
  L^*_{T-} - L^*_{S-} &\leq 
  \bigl( \sup_{s\in[S,T]} \norm{L_s} - \sup_{s < S} \norm{L_s} \bigr)
  \mathbf{1}_{\{S<T\}}\\
  &\leq \bigl( \sup_{s\in[S,T]} \norm{L_s} - \norm{L_{S-}} \bigr)
  \mathbf{1}_{\{S<T\}}\\
  &\leq \bigl( \sup_{s \in [S,T]} \norm{L_s - L_{S-}}\bigr) 
   \mathbf{1}_{\{S<T\}}
  = (L^{S,T})_\infty^* \mathbf{1}_{\{S<T\}},
  \end{align*}
  where we have used the obvious estimates $\sup_{s<S} \norm{L_s} \geq
  \norm{L_{S-}}$ and $\norm{x}-\norm{y} \leq \norm{x-y}$. 

  \smallskip

  \textsc{Step 4.} The previous step immediately implies
  \[
  \E \bigl( L_{T-}^* - L_{S-}^* \bigr)^q 
  \lesssim \E \bigl( (L^{S,T})_\infty^* \bigr)^q \mathbf{1}_{\{S<T\}},
  \]
  By Theorem \ref{thm:iBDG} we have
  \[
  \E \bigl( (L^{S,T})_\infty^* \bigr)^q \mathbf{1}_{\{S<T\}} \lesssim_q
  \E \norm[\big]{[L^{S,T},L^{S,T}]_\infty^{1/2}}_{L_q}^q,
  \]
  where
  \[
  [L^{S,T},L^{S,T}]_\infty \leq [L,L]_T + \norm{\Delta L_S}^2
  \leq [L,L]_{T-} + \norm{\Delta L_S}^2 + \norm{\Delta L_T}^2,
  \]
  hence
  \[
  \norm[\big]{[L^{S,T},L^{S,T}]_\infty^{1/2}}_{L_q} \leq
  \norm[\big]{[L,L]_{T-}^{1/2}}_{L_q} + \norm{\Delta L_S}_{L_q}
  + \norm{\Delta L_T}_{L_q}
  \lesssim \norm[\big]{[L,L]_{T-}^{1/2}}_{L_q} + S_{T-},
  \]
  therefore also, recalling that $q \leq 2$,
  \[
  \E \bigl( L_{T-}^* - L_{S-}^* \bigr)^q \lesssim_q
  \E \norm[\big]{[L,L]_{T-}^{1/2}}_{L_q}^q + \E S_{T-}^q \lesssim_q
  \E \norm[\big]{\ip{L}{L}_{T-}^{1/2}}_{L_q}^q + \E S_{T-}^q.
  \]
  We can now apply the previous Lemma to obtain
  \[
  \E (L^*_\infty)^p \lesssim_{p,q}
  \E \norm[\big]{\ip{L}{L}_\infty^{1/2}}_{L_q}^p + \E S_\infty^p.
  \]

  \smallskip

  \textsc{Step 5.} Recalling the estimate on $K$, we are left with
  \[
  \E (M^*_\infty)^p \lesssim
  \E \norm[\big]{\ip{L}{L}_\infty^{1/2}}_{L_q}^p + \E S_\infty^p,
  \]
  where, since $L=(g-g') \star \m$ and $|g'| \leq |g|$ pointwise (in
  fact by construction $g'$ represents some jumps of $M$ only),
  \[
  \ip{L}{L}_\infty = \int |g-g'|^2\,d\nu \lesssim \int |g|^2\,d\nu,
  \]
  and
  \[
  \E S_\infty^p = \E \sup_{s\geq 0} \norm{\Delta M_s}^p_{L_q} 
  \leq \E \sum \norm{\Delta M}^p_{L_q}
  = \E \int \norm{g}_{L_q}^p\,d\mu
  = \E \int \norm{g}_{L_q}^p\,d\nu.
  \qedhere
  \]
\end{proof}


\subsection{Case $2 \leq p \leq q$}
The upper bound in Theorem \ref{thm:m} with parameters $p$ and $q$
such that $2 \leq p \leq q$ follows by Proposition \ref{prop:bella}
and by the next Proposition.

\begin{prop}
  Let $2 \leq p \leq q$. Then one has
  \[
  \E\sup_{t \geq 0} \norm[\big]{(g \star \m)_t}^p_{L_q} \lesssim_{p,q}
  \E\int \norm[\big]{g}_{L_q}^p\,d\nu
  + \E\norm[\bigg]{\biggl( \int |g|^2\,d\nu\biggr)^{1/2}}_{L_q}.
  \]
\end{prop}
\begin{proof}
  Recalling Theorem \ref{thm:iBDG}, one has
  \[
  \E\norm[\big]{(g \ast \m)^*_\infty}^p_{L_q} \lesssim_{p,q}
  \E\norm[\big]{[M,M]^{1/2}_\infty}^p_{L_q} 
  = \E\norm[\bigg]{\biggl( \int |g|^2\,d\mu \biggr)^{1/2}}^p_{L_q}.
  \]
  If $p=2$, it holds
  \begin{align*}
    \E\norm[\bigg]{\biggl( \int |g|^2\,d\mu \biggr)^{1/2}}^2_{L_q}
    &= \E\norm[\bigg]{\int |g|^2\,d\mu }_{L_{q/2}}\\
    &\leq \E\int \norm[\big]{|g|^2}_{L_{q/2}}\,d\mu 
    = \E\int \norm[\big]{g}^2_{L_q}\,d\mu 
    = \E\int \norm[\big]{g}^2_{L_q}\,d\nu,
  \end{align*}
  that is the claim is proved in the case $p=2$.
  Therefore we assume, for the rest of the proof, that $p>2$. One has
  \[
  \E\norm[\bigg]{\biggl( \int |g|^2\,d\mu \biggr)^{1/2}}^p_{L_q}
  \leq \E\norm[\bigg]{\bigg\lvert \int |g|^2\,d\m \bigg\rvert^{1/2}}^p_{L_q}
  + \E\norm[\bigg]{\biggl( \int |g|^2\,d\nu \biggr)^{1/2}}^p_{L_q},
  \]
  where
  \[
  \E\norm[\bigg]{\bigg\lvert \int |g|^2\,d\m \bigg\rvert^{1/2}}^p_{L_q}
  = \E\norm[\bigg]{\int |g|^2\,d\m}^{p/2}_{L_{q/2}}.
  \]
  If $2<p \leq 4$, i.e. if $1 < p/2 \leq 2$, Propositions
  \ref{prop:13} and \ref{prop:due} imply that
  \[
  \E\norm[\bigg]{\int |g|^2\,d\m}^{p/2}_{L_{q/2}}
  \lesssim_{p,q} \E \int \norm[\big]{\,|g|^2\,}^{p/2}_{L_{q/2}}\,d\nu
  = \E\int \norm[\big]{g}^p_{L_q}\,d\nu.
  \]
  We have thus proved that the claim of the Proposition is true for
  any $p$, $q$ such that $2 < p \leq 4$ and $q \geq p$. We now proceed
  by induction, i.e. we assume the claim is true for $2 < p \leq 2^n$
  (which we have just verified for $n=2$), and we show that the claim
  is true for $2 < p \leq 2^{n+1}$. In fact, we have just seen that
  \[
  \E\norm[\big]{(g \ast \m)^*_\infty}^p_{L_q} \lesssim_{p,q} 
  \E\norm[\bigg]{\int |g|^2\,d\m}^{p/2}_{L_{q/2}}
  + \E\norm[\bigg]{\biggl( \int |g|^2\,d\nu \biggr)^{1/2}}^p_{L_q},
  \]
  where, since $p/2 \leq 2^n$, the inductive assumption implies
  \begin{align*}
  \E\norm[\bigg]{\int |g|^2\,d\m}^{p/2}_{L_{q/2}} &\lesssim_{p,q}
  \E\norm[\bigg]{\biggl( \int |g|^4\,d\nu \biggr)^{1/2}}^{p/2}_{L_{q/2}}
  + \E\int \norm[\big]{|g|^2}^{p/2}_{L_{q/2}}\,d\nu\\
  &= \E\norm[\bigg]{\biggl( \int |g|^4\,d\nu \biggr)^{1/4}}^p_{L_q}
  + \E\int \norm[\big]{g}^p_{L_q}\,d\nu.
  \end{align*}
  Note that, by Lemma \ref{lm:stimette}, since $4 < p$, one has
  \[
  \biggl(\int |g|^4\,d\nu \biggr)^{1/2} = \norm[\big]{g}_{L_4(\nu)}
  \leq \norm[\big]{g}_{L_2(\nu)} + \norm[\big]{g}_{L_p(\nu)},
  \]
  hence
  \[
  \E\norm[\bigg]{\biggl( \int |g|^4\,d\nu \biggr)^{1/4}}^p_{L_q}
  \lesssim
  \E\norm[\bigg]{\biggl( \int |g|^2\,d\nu \biggr)^{1/2}}^p_{L_q}
  + \E \norm[\Big]{\norm[\big]{g}_{L_p(\nu)}}^p_{L_q}.
  \]
  The proof is finished by observing that, since $q \geq p$,
  H\"older-Minkowski's inequality \eqref{eq:HM} yields
  \[
  \E \norm[\Big]{\norm[\big]{g}_{L_p(\nu)}}^p_{L_q} \leq 
  \E \norm[\Big]{ \norm[\big]{g}_{L_q} }^p_{L_p(\nu)}
  = \E \int \norm[\big]{g}^p_{L_q}\,d\nu.
  \qedhere
  \]
\end{proof}


\subsection{Case $2 \leq q \leq p$}
The upper bound in Theorem \ref{thm:m} with parameters $p$ and $q$
such that $2 \leq p \leq q$ is proved in the following
Proposition.
\begin{prop}
  Let $2 \leq q \leq p$. Then one has
  \[
  \E\sup_{t \geq 0} \norm[\big]{(g \star \m)_t}_{L_q} \lesssim_{p,q}
  \E\int \norm[\big]{g}^p_{L_q}\,d\nu
  + \E\biggl( \int \norm[\big]{g}^q_{L_q}\,d\nu \biggr)^{p/q}
  + \E\norm[\bigg]{\biggl( \int |g|^2\,d\nu \biggr)^{1/2}}^p_{L_q}
\]
\end{prop}
\begin{proof}
  The claim is certainly true if $q=2$, by Theorem \ref{thm:H}. We
  shall therefore assume $q>2$ from now on. In view of Theorem
  \ref{thm:iBDG}, it is enough to estimate the $\L_pL_q$ norm of
  $[M,M]_\infty^{1/2}$. We use once again the decomposition
  \[
  [M,M]_\infty^{1/2} = \biggl( \int |g|^2\,d\mu \biggr)^{1/2}
  \leq \biggl| \int |g|^2\,d\m \biggr|^{1/2}
  + \biggl( \int |g|^2\,d\nu \biggr)^{1/2},
  \]
  and its immediate consequence
  \[
  \norm[\big]{[M,M]_\infty^{1/2}}_{\L_pL_q} \leq
  \norm[\bigg]{\biggl| \int |g|^2\,d\m \biggr|^{1/2}}_{\L_pL_q} 
  + \norm[\bigg]{\biggl( \int |g|^2\,d\nu \biggr)^{1/2}}_{\L_pL_q}.
  \]
  We are thus left with the task of estimating the first term on the
  right-hand side. In particular, writing
  \begin{equation}     \label{eq:tardi}
  \E\norm[\bigg]{\biggl| \int |g|^2\,d\m \biggr|^{1/2}}^p_{L_q}
  = \E\norm[\bigg]{\int |g|^2\,d\m}^{p/2}_{L_{q/2}},
  \end{equation}
  we observe that, assuming $2 < q \leq 4$, i.e. $1<q/2 \leq 2$,
  Propositions \ref{lm:34} and \ref{lm:38} imply
  \begin{align*}
  \E\norm[\bigg]{\biggl| \int |g|^2\,d\m \biggr|^{1/2}}^p_{L_q}
  &\lesssim_{p,q} 
   \E\biggl( \int \norm[\big]{|g|^2}^{q/2}_{L_{q/2}}\,d\nu \biggr)^{p/q}
   + \E\int \norm[\big]{|g|^2}^{p/2}_{L_{q/2}}\,d\nu\\
  &= \E\biggl( \int \norm[\big]{g}^q_{L_q}\,d\nu \biggr)^{p/q}
   + \E\int \norm[\big]{g}^p_{L_q}\,d\nu.
  \end{align*}
  We have thus proved that the claim is claim is true for any $q \in
  [2,4]$. We proceed by induction, showing that if the claim is true
  for $2 \leq q \leq 2^n$ (which is indeed the case with $n=2$), then
  the claim remains true for $2 \leq q \leq 2^{n+1}$. In view of the
  reasoning at the beginning of the proof, it suffices to estimate the
  term on the right-hand side of \eqref{eq:tardi}: one has
  \begin{align*}
   \E\norm[\bigg]{\int |g|^2\,d\m}^{p/2}_{L_{q/2}} &\lesssim_{p,q} 
   \E\biggl( \int \norm[\big]{|g|^2}^{q/2}_{L_{q/2}}\,d\nu \biggr)^{p/q}
   + \E\int \norm[\big]{|g|^2}^{p/2}_{L_{q/2}}\,d\nu
   + \E\norm[\bigg]{\biggl( \int |g|^4\,d\nu \biggr)^{1/2}}^{p/2}_{L_{q/2}}\\
  &= \E\biggl( \int \norm[\big]{g}^q_{L_q}\,d\nu \biggr)^{p/q}
   + \E\int \norm[\big]{g}^p_{L_q}\,d\nu
   + \E\norm[\bigg]{\biggl( \int |g|^4\,d\nu \biggr)^{1/4}}^p_{L_q}
  \end{align*}
  Since $4 < q$, Lemma \ref{lm:stimette} yields 
  \[
  \biggl( \int |g|^4\,d\nu \biggr)^{1/4} = \norm{g}_{L_4(\nu)}
  \leq \norm{g}_{L_2(\nu)} + \norm{g}_{L_q(\nu)},
  \]
  hence also
  \begin{align*}
    \E\norm[\bigg]{\biggl( \int |g|^4\,d\nu \biggr)^{1/4}}^p_{L_q}
    &\leq \E\norm*{\norm[\big]{g}_{L_2(\nu)}}^p_{L_q} 
    + \E\norm*{\norm[\big]{g}_{L_q(\nu)}}^p_{L_q}\\
    &= \E\norm[\bigg]{\biggl( \int |g|^2\,d\nu \biggr)^{1/2}}^p_{L_q}
    + \E\norm*{\norm[\big]{g}_{L_q}}^p_{L_q(\nu)}\\
    &= \E\norm[\bigg]{\biggl( \int |g|^2\,d\nu \biggr)^{1/2}}^p_{L_q}
    + \E\biggl( \int \norm[\big]{g}^q_{L_q}\,d\nu \biggr)^{p/q},
  \end{align*}
  thus concluding the proof.
\end{proof}

\let\oldbibliography\thebibliography
\renewcommand{\thebibliography}[1]{%
  \oldbibliography{#1}%
  \setlength{\itemsep}{-1pt}%
}

\bibliographystyle{amsplain}
\bibliography{ref}

\end{document}